\documentclass[12pt]{article}
\usepackage{amsfonts,mathrsfs,amssymb,amsthm,mathptm}
\usepackage{amsmath,amscd}
\usepackage{mathptm,pslatex}
\usepackage{multirow}
\usepackage{color}
\usepackage[dvips]{graphicx}
\usepackage[all]{xy}
\usepackage{graphicx}
\usepackage{fancyhdr}
\usepackage{url}
\usepackage{cite}
\usepackage{exscale}
\usepackage{relsize}
\usepackage{bbm}
\theoremstyle{plain}
\newtheorem{theorem}{Theorem}[section]
\newtheorem{prop}[theorem]{Proposition}
\newtheorem{lem}[theorem]{Lemma}

\numberwithin{equation}{section}
\allowdisplaybreaks

\begin{document}

\begin{center}
{\bf\Large On finite totally $k$-closed groups }
\end{center}
\vskip 3mm
\begin{center}
{ Jiawei He$^{a,*}$, Xiaogang Li $^{b}$}
\end{center}

\vskip 3mm

\begin{abstract}
Let $G$ be a finite group acting faithfully on a finite set $\Omega$. For a positive integer $k$, $G$ acts naturally on the Catesian product $\Omega^k := \Omega \times ...\times \Omega$. In this paper, we prove that finite nilpotent group $G$ with $2\nmid |G|$ is a totally $k$-closed group if and only if $G$ is abelian with $n(G)\leq k-1$ or cyclic, where $n(G)$ is the number of invariant factors in the invariant factor decomposition of $G$.

\end{abstract}
\vskip 3mm


\renewcommand{\thefootnote}{\empty}
\footnotetext{School of Mathematics and Information Science, Nanchang Hangkong University, Nanchang, China}
\footnotetext{$^*$corresponding author, School of Mathematics and Information Science, Nanchang Hangkong University, Nanchang, China}
\footnotetext{ Email:  hjwywh@mails.ccnu.edu.cn }

\section{Introduction}
For a positive integer $k$, let $G$ be a permutation group on a finite set $\Omega$. $G$ is said to be $k$-$closed$ if $G$ is the largest subgroup of $\mathrm{Sym}(\Omega)$ which leaves invariant each of the $G$-orbits in the induced action on $\Omega \times ...\times \Omega=\Omega^k $. A finite group $G$ is said to be a $totally$ $k$-closed group if $G = G^{(k),\Omega}$ for any faithful $G$-set $\Omega$.

\medskip

 A key problem is to determine which groups are totally $k$-closed. In recent years, finite totally $k$-closed groups have been studied in several papers. In 2018, A. Abdollahi e al. \cite{M14} first studied which finite nilpotent groups have the same faithful permutation representations as their $2$-closures. They prove that a finite nilpotent group is totally $2$-closed if and only if it is cyclic or a direct product of a generalized quaternion group with a cyclic group of odd order. Recently, D. Churikov and I. Ponomarenko explain that a finite nilpotent permutation group is $2$-closed if and only if every Sylow subgroup of its group is $2$-closed.  In \cite{5}, for a finite abelian group $G$, the minimal positive integer $k$ for which $G$ is totally $k$-closed is given. Actually, the minimal positive integer $k$ is proved to be $1$ plus the number of invariant factors of $G$.
\medskip

\vskip 3mm

In this short note, we will focus on finite totally $k$-closed groups. Let $G$ be a finite group acting faithfully on a finite set $\Omega$. Suppose that the size of $\Omega$ is $n$, the largest $p$-power divisor of $n$ is denoted by $n_{p};$ if $\pi$ is a set of prime divisors of $G,$ then we put $n_{\pi}:=\prod_{p \in \pi} n_{p}.$  If $G$ is abelian, then we can write $G=G_{1} \times \cdots \times G_{m}$ for some $m \geq 1$ such that each $G_{i} \cong \mathbb{Z}_{d_{i}},~ d_{1}>1,$ and $d_{i} \mid d_{i+1}$ for $1 \leq i<m$. The $G_{i}$ are called the {\it invariant factors} of $G$, and we write $n(G)$ for the number of invariant factors. The following is the main result of the present paper.

\medskip

\begin{theorem}
Let $G$ be a finite nilpotent group with $2\nmid |G|$. Then $G$ is a totally $k$-closed group if and only if $G$ is abelian with $n(G)\leq k-1$ or cyclic.
\end{theorem}
\vskip 3mm

 For the reader convenience that some related concepts is given in Section 2. The proof of Theorem 1.1 is given in Section 3.\\

\vskip 3mm
$\mathbf{Notation.}$ Throughout the paper, $\Omega$ denotes a finite set and the symmetric
group of $\Omega$ is denoted by $\mathrm{Sym}(\Omega)$. Let $G$ be a subgroup of the permutation group $\operatorname{Sym}(\Omega)$, we use the symbol
$$
\operatorname{Orb}_k(G):=\left\{(\alpha_1, \alpha_2,...,\alpha_k)^{G}: (\alpha_1, \alpha_2,...,\alpha_k) \in \Omega^k\right\}
$$
for the set of $G$-orbits $O_{(\alpha_1, \alpha_2,...,\alpha_k)}:=\left\{(\alpha_1, \alpha_2,...,\alpha_k)^{g}: g \in G\right\},$ where $(\alpha_1, \alpha_2,...,\alpha_k)^{g}=(\alpha_1^g, \alpha_2^g,...,\alpha_k^g).$

For $\Delta \subseteq \Omega$, set $$G_{\{\Delta\}}:=\{g\in G~:~\Delta^g=\Delta\},\quad G_\Delta:=\{g\in G~:~\alpha^g=\alpha~for~all~\alpha\in \Delta\}.$$

\vskip 3mm
\section{Preliminary}\label{Pr}
\quad\;
Let $G$ be a finite group and $G$ acts faithfully on a set $\Omega$.  Then $G^{\Delta}$ denotes the permutation group induced by this action, where $\Delta\subseteq\Omega.$ Let $\pi(G)$ is a set of prime divisors of $G$. For
$p\in \pi(G)$, we denote by $\mathrm{Syl}_p(G)$ the set of Sylow p-subgroups of $G$. We denote by $\mathrm{Orb}(G)$ the set of $G$-orbits in $\Omega$. Our notations here are standard and they can be found in \cite{3,P15}.
\vspace{3mm}

The following two lemmas will be used in the proof of the main theorem.
\vspace{3mm}

\begin{lem}  (\cite[Theorem 5.6]{W49}) Let $G \leq \mathrm{Sym}(\Omega)$, let $k \geq1$,
and let $x \in \mathrm{Sym}(\Omega)$. Then $x\in G^{(k),\Omega}$ if and only if, for all $(\alpha_1,...,\alpha_k) \in \Omega^k$,
there exists $g \in G$ such that $(\alpha_1,...,\alpha_k)^x=(\alpha_1,...,\alpha_k)^g$.
\end{lem}
\vspace{3mm}

\begin{lem} (\cite[Theorem 1.1]{5}) Let $G$ be a nilpotent and acts on a finite set $\Omega$ faithfully. Then $=G^{(2),\Omega}$ is nilpotent. Moreover,
$$
G^{(2),\Omega}=\prod_{P \in \operatorname{Syl}(G)} {P}^{(2),\Omega}
.$$
\end{lem}
\vspace{3mm}

\medskip

$\mathbf{Remark.}$ For a permutation group $G \leq \operatorname{Sym}(\Omega)$, as is well known,
$$
G \leq G^{(k), \Omega} \leq G^{(k-1), \Omega}\leq ...\leq G^{(2), \Omega}.
$$ Thus $G^{(k),\Omega}$ is nilpotent and $\pi(G)=\pi(G^{(k),\Omega})=\pi (G^{(2),\Omega})$ by Lemma 2.2 whenever $G$ is nilpotent.

\medskip

The following two lemmas are general results about finite nilpotent groups.

\medskip
\begin{lem} (\cite[Lemma 2.4]{5}) Let $G$ be a finite nilpotent permutation group and $P$ is a Sylow $p$-subgroup of $G.$ Let $k$ be a positive integer. Let $\Delta_{1}, \ldots, \Delta_{k} \in \mathrm{Orb}(P),$ $\Delta=\bigcup_{i=1}^{k} \Delta_{i},$ and $L$ be the subgroup of $G$ consisting of all elements fixing each $\Delta_{i}$ setwise. Then $L^{\Delta}=P^{\Delta}$. \end{lem}

\medskip

\begin{lem} (\cite[Lemma 3.1]{6}) Assume that $G$ is nilpotent and acts transitively on $\Omega$. Let $H$ be a Hall subgroup of $G$. Then

(1) the size of every $H$-orbit is equal to $n_{\pi},$ where $n=|\Omega|$ and $\pi=\pi(G)$,

(2) $G$ acts on $\mathrm{Orb}(H) ;$ moreover, the kernel of this action is equal to $\mathrm{H}$. \end{lem}

\medskip

The following result is known whenever $k=2$.  Its proof is
analogous to that of (\cite[Theorem 1.3]{5}).

\medskip
\begin{lem} Let $G$ be a finite nilpotent permutation group on $\Omega$, and let $k \geq 2$, $p \in \pi\left(G^{(k), \Omega}\right), P \in \operatorname{Syl}_{p}(G)$ and $Q \in \operatorname{Syl}_{p}\left(G^{(k), \Omega}\right) .$ Then $P^{(k), \Omega}=Q$.
\end{lem}

\medskip

Let $G$ be an arbitrary
group with a normal subgroup $K$. Define a mapping $$\psi: G \rightarrow G/K$$ be a homomorphism
of $G$ onto $G/K$ with kernel $K$. We use notation $$T:=\{1,g_1,g_2,...,g_n \mid {g_i}\in G\}$$ which is a set of right coset representatives
of $K$ in $G$. Let $x \in G$ and $$f_{x}: G/K\rightarrow K,~~K\mapsto xg_k^{-1},~ {g_i}K\mapsto {g_i}x{g_j}^{-1},$$ where $g_i,g_j,g_k\in T$, $g_k K=x K$ and $g_j K=g_i x K$.\\

The following result plays a key role in the whole theory.

\begin{lem}(Universal embedding theorem \cite[Theorem 2.6 A]{P15}) Under the notation above, then $\varphi(x):=\left(f_{x}, \psi(x)\right)$ defines an embedding $\varphi$ of $G$ into $K\wr G/K$. Furthermore, if $K$ acts faithfully on a set $\Delta$, then $G$ acts faithfully on
$\Delta \times G/K$ by the rule $(\alpha, g_iK)^{x}=\left(\alpha^{f_{x}(g_iK)}, g_i \psi(x)\right)=(\alpha^{{g_i}x{g_j}^{-1}}, g_j K)$, where $\alpha\in \Delta$ and $g_j K=g_i x K$. \end{lem}

\medskip
\section{Proof of the Main Result}

Throughout let $G$ denote an arbitrary finite group with identity element $1$ and $\mathbb{Z}$ the integer ring. The set of Sylow subgroups of $G$ is
denoted by Syl(G).

\quad\;

\begin{theorem} Let $G$ be a finite nilpotent permutation group on $\Omega$ and $k \geq 2$, then $G^{(k), \Omega}=\prod_{P \in \operatorname{Syl}(G)} P^{(k), \Omega}$.
\end{theorem}
\begin{proof}  By the Remark, $\pi\left(G^{(k), \Omega}\right)=\pi(G)$. Then Lemma 2.5 implies that the unique Sylow $p$-subgroup of $G^{(k), \Omega}$ is equal to $P^{(k), \Omega}$, where $P$ is the unique Sylow $p$-subgroup of $G.$ Hence $$G^{(k), \Omega}=\prod_{P \in \operatorname{Syl}(G)} P^{(k), \Omega},$$ as required.
\end{proof}
\medskip

The following two propositions which are very important for proving Theorem 1.1.
\medskip

\begin{prop}
 Let $G$ be a finite $p$-group. Suppose $G$ has  a normal subgroup $H$  that satisfies $H\cong \mathbb Z_p\times\mathbb Z_p$ and $|H\cap Z(G)|=p$. Then $G$ is not totally $k$-closed group for any $k\in \mathbb{Z}$.
\end{prop}

\begin{proof}  We use counterexample argument to prove this theorem. Assume that $G$ is totally $k$-closed, and let $H=\langle a\rangle\times\langle c\rangle$, where $a\in Z(G)$.
Let $\Omega$ be a faithful $G$-set and $a, c\in \mathrm{Sym}(\Omega)$ be pairwise independent cycles on $\Omega$.
Without loss of generality, we may assume that $$H=\langle a,c\rangle=\langle(1,2,...,p),(p+1,p+2,...,2p)\rangle,$$ where $a=(1,2,...,p)$, $c=(p+1,p+2,...,2p)$ and $\{1,2,...,2p\}\subseteq \Omega$. Write $O_c$ for the orbit containing $c$ with respect to the conjugation by $G$. Observe that $|G:C_G(c)|=|O_c|\leq p(p-1)$, which implies that $$|G:C_G(H)|=|G:C_G(c)|=p.$$ Thus there exists
$b \in G \backslash C_G(H)$ such that $G=\langle b\rangle C_G(H)$ and $b^{p} \in C_G(H)$. Define $$\Delta:=\{1,2,...,p,p+1,...,2p\}.$$ The Universal
embedding theorem states clearly that $C_G(H)$ embeds in $H\wr C_G(H)/H \leq \operatorname{Sym}(\Delta \times C_G(H)/H)$. Let $gH \in C_G(H)/H$. Then $$(i, gH)^{h}=(i^{h},gHhH)=\left(i^{h}, gH\right)$$ for
each $i \in\{1,2,...,p,p+1,...,2p\}$ and $h \in H$.

\medskip
Now, again by Universal embedding theorem, $G$ embeds in $C_G(H)\wr G / {C_G(H)} \leq \operatorname{Sym}(\Gamma \times G /{C_G(H)})$,
where $\Gamma:=\Delta \times {C_G(H)} / H$. Let $\{1, b, b^2,...,b^{p-1}\}$ be the set of right transversal of ${C_G(H)}$ in $G$. Therefore, for each $\gamma \in \Gamma$ and $g\in b^i{C_G(H)}$,
$$
\begin{array}{ll}
(\gamma, {C_G(H)})^{g}=\left\{\begin{array}{ll}
\left(\gamma^{g}, {C_G(H)}\right), &i=0, \\\\
\left(\gamma^{g b^{-i}}, b^i{C_G(H)}\right), & i\neq 0 ,
\end{array}\right. \\$$
\\

$$(\gamma, b^j{C_G(H)})^{g}=\left\{\begin{array}{ll}
\left(\gamma^{{b^j}g b^{-j}}, b^j{C_G(H)}\right), &i=0,\\\\
\left(\gamma^{b^j gb^{-i-j}}, b^{i+j}{C_G(H)}\right), & i\neq 0.
\end{array}\right.
\end{array}
$$
 Then for each $xH \in {C_G(H)}/ H,$ it follows that
$$
G_{((1, xH), {C_G(H)})}=G_{((2, xH), {C_G(H)})}=...=G_{((p, xH), {C_G(H)})}=\langle c\rangle,$$$$G_{((1, xH), b{C_G(H)})}=G_{((2, xH), b{C_G(H)})}...=G_{((p, xH), b{C_G(H)})}=\langle  c^b\rangle.
$$
Moreover,
$$
G_{((p+1, xH), {C_G(H)})}=G_{((p+1, xH), b{C_G(H)})}=G_{((p+2, xH), {C_G(H)})}$$$$=G_{((p+2, xH), b{C_G(H)})}=...=G_{((2p, xH), {C_G(H)})}=G_{((2p, xH), b{C_G(H)})}=\langle a\rangle.
$$ If $c^b=c^i$ for some $1<i<p$, then we may replace $c$ with $ca$. This implies that $$H=\langle ca\rangle\times \langle a\rangle~ and ~{ca}^b=c^ia\notin \langle ca\rangle.$$ Hence we may assume $\langle c\rangle $ is not a normal subgroup of $G$. In this case, we may assume that $\Omega=\Gamma \times G /{C_G(H)}$.

 Define the map $\theta$ on $\Omega$ as follows:
$$
((i, xH), b^m{C_G(H)}) \mapsto\left\{\begin{array}{ll}
((i,xH), b^m{C_G(H)}) & i=1,2,...,p ,\\\\
((i+1, xH), b^m{C_G(H)}) & p+1\leq i<2p, \\\\
((p+1, xH), b^m{C_G(H)}) & i=2p,
\end{array}\right.
$$ where $xH \in {C_G(H)} / H$, $b^m{C_G(H)} \in G /{C_G(H)}$.
Then clearly $1 \neq \theta \in \operatorname{Sym}(\Omega)$. Next, we will prove that $\theta \in G^{(k), \Omega}$. Let $$\alpha_1=((i_1, x_1H), b^{{m}_1}{C_G(H)}),
\alpha_2=\left(\left(i_2, x_2H\right), b^{{m}_2}{C_G(H)}\right) ,..., \alpha_k=\left(\left(i_k, x_kH\right), b^{{m}_k}{C_G(H)}\right) ,$$ where $\{i_1,i_2,...,i_k\}\subseteq \Delta$. Then
$$
(\alpha_1, \alpha_2,...,\alpha_k)^{\theta}=\left\{\begin{array}{ll}
(\alpha_1, \alpha_2,...,\alpha_k), & \mathrm{if} ~i_1, i_2,...,i_k \in\{1,2,...,p\}, \\\\
(\alpha_1, \alpha_2,...,\alpha_k)^{c}, & \mathrm{otherwise}. \\
\end{array}\right.
$$
Thus $\theta \in G^{(k), \Omega}$. Observe that $\theta$ fixes $((1, H), {C_G(H)})$ and $((1, H), b{C_G(H)})$. It follows that $$\theta\in G^{(k),\Omega}_{((1, H), {C_G(H)})}\cap G^{(k),\Omega}_{((1, H), b{C_G(H)})}.$$ Also, $$G_{((1, H), {C_G(H)})} \cap G_{((1, H), b{C_G(H)})}= \langle c\rangle \cap\langle c^b\rangle=1,$$ a contradiction. This completes the proof of this proposition.
\end{proof}

\medskip

\begin{prop} Let $G$ be a finite totally $k$-closed group, then $Z(G)$ is totally $k$-closed.
\end{prop}

\begin{proof} Let $\Delta$ be a set on which $Z(G)$ acts faithfully. Set $\Gamma:=G/Z(G)$ and
$\Omega:=\Delta \times \Gamma$. Then $G$ acts faithfully on $\Omega$ by the Universal Embedding Theorem. Thus $G^{(k),\Omega}=G$. As $[Z(G),G]=1$, it follows that $[Z(G)^{(k),\Omega},G^{(k),\Omega}]=1$ by \cite[Exercise 5.29]{W49}, which implies that $$
Z(G) \leq Z(G)^{(k), \Omega} \leq Z\left(G^{(k), \Omega}\right)=Z(G).
$$ Therefore $Z(G)$ is $k$-closed over $\Omega$.

\medskip
Let $\varphi\in Z(G)^{(k), \Delta}$. Then, by Lemma 2.1, there exists $c\in Z(G)$ such that $$(\alpha_1,\alpha_2,...,\alpha_k)^\varphi=(\alpha_1,\alpha_2,...,\alpha_k)^c$$ for each $\alpha_1,\alpha_2,...,\alpha_k\in \Delta$. Define $$\bar{\varphi}: \Omega \rightarrow \Omega,\quad ~~(\alpha, m)\mapsto\left(\alpha^{\varphi}, m\right).$$
Then $\bar{\varphi} \in \operatorname{Sym}(\Omega)$ and for all $(\alpha_1, m_1),(\alpha_2, m_2),...,(\alpha_k, m_k)\in \Omega$, we have that $$
\begin{aligned}
((\alpha_1, m_1),(\alpha_2, m_2),...,(\alpha_k, m_k))^{\bar{\varphi}} &=(({\alpha_1}^\varphi, m_1),({\alpha_2}^\varphi, m_2),...,({\alpha_k}^\varphi, m_k)) \\
&=(({\alpha_1}^c, m_1),({\alpha_2}^c, m_2),...,({\alpha_k}^c, m_k))\\
&=((\alpha_1, m_1),(\alpha_2, m_2),...,(\alpha_k, m_k))^{c}
\end{aligned}
$$
with respect to this embedding. Thus $\bar{\varphi} \in Z(G)^{(k), \Omega}$. Hence $\bar{\varphi} \in Z(G)$ and so $\bar{\varphi}=c'$ for some $c' \in Z(G)$. Therefore for each
$\alpha \in \Delta$ and $m \in \Gamma$, $$(\alpha, m)^{\bar{\varphi}}=(\alpha, m)^{c'},$$ which implies that $ \alpha^{\varphi}=\alpha^{c'}$ for each $\alpha \in \Delta$. Hence
$\varphi=c' \in Z(G),$ forcing $Z(G)$ to be $k$-closed on $\Delta,$ as required.
\end{proof}

\medskip
The proof of the following two lemmas comes from \cite[Theorem 1.1]{6} and \cite[Theorem A]{4}, respectively.

\medskip
\begin{lem} Let $G$ be a finite abelian group with $|G|>1$. Then $G$ is totally $(n(G)+1)$-closed but not totally $n(G)$-closed.
\end{lem}

\begin{lem} A finite nilpotent group is totally 2-closed if and only if it is cyclic or a direct
product of a generalized quaternion group with a cyclic group of odd order.
\end{lem}

\medskip
Now, let us start proving Theorem 1.1. When $G$ is a finite $p$-group, this main theorem holds from the following theorem.
\medskip

\begin{theorem} Let $G$ be a finite $p$-group with $p\neq 2$. Then $G$ is a totally $k$-closed group if and only if $G$ is an abelian $p$-group with $n(G)\leq k-1$ or cyclic.
\end{theorem}

\begin{proof} We first prove the "only if" direction. If $G$ is a totally $k$-closed group, then $Z(G)$ is totally $k$-closed by Proposition 3.3. Then we easily deduce the results by Lemma 3.4 whenever $G$ is abelian , thus we may assume that $G$ is non-abelian. Using induction on $|G|$, then $Z(G)$ is an abelian $p$-group with $n(G)\leq k-1$ or cyclic by the inductive hypothesis. 

\medskip

We next consider the two cases for $Z(G)$.
\medskip

{\bf Case 1.}\, Suppose that $Z(G)$ is cyclic. Then $G$ has at least one abelian normal subgroup of type $(p,p)$. We take $H$ to be an abelian normal subgroup of type $(p,p)$ of $G$. In particular, we may assume that $|Z(G)\cap H|=p$. Otherwise, we can choose $N\leq H$ such that $$N\unlhd G~ ~\mathrm{and} ~~|N|=p.$$  We can now substitute $\langle a\rangle\times N$ for $H$, where $a\in Z(G)$ and $o(a)=p$. Thus $G$ is not $k$-closed by Proposition 3.2, a contradiction.
\medskip

{\bf Case 2.}\, Suppose that $Z(G)$ is not cyclic. Similarly, we still have that there exists an abelian normal subgroups of type $(p,p)$ of $G$. Moreover, we may assume that all abelian normal subgroups of type $(p,p)$ of $G$ are contained in $Z(G)$. Let $$1\unlhd H_1\unlhd H_2...\unlhd  H_{n-1}\unlhd H_{n}=G$$ be a normal series for which each $|H_{i+1}/H_{i}|\leq p$ and $H_2$ to be of type $(p, p)$ here. Note that $H_{3}$ is abelian because $|H_{3}/H_2|\leq p$ and $H_2\leq Z(G)$. Assume that $\Omega$ is a set such that $G$ acts on $\Omega$ faithfully, then $H_{3}^{(k),\Omega}$ is abelian by (\cite[Exercise 5.29]{W49}).

\medskip
If $H_{3}^{(k),\Omega}> H_{3}$.  Then, by (\cite[Lemma 1.3]{M14}), $$x^{-1}H_{3}^{(k),\Omega}x=(x^{-1}H_{3}x)^{(k),\Omega}=H_{3}^{(k),\Omega}$$ for any $x\in G$.
Hence we may assume that there exists $H_{3}^{(k),\Omega}=H_j$ for some $4 \leq j\leq n$. Obviously, $H_j=H_{j}^{(k),\Omega}$. We are done from Lemma 3.5 whenever $j=n$.

Now, assumption that $$H_{3}^{(k),\Omega}= H_{3}<G ~~\mathrm{or}~~ H_{3}<H_{3}^{(k),\Omega}=H_j<G.$$ Let $\Delta$ be a set so that $H_{3}$ acts faithfully. Let $\Omega'=\Delta\times G/H_{3}$, and so $G$ acts faithfully on $\Omega'$. 
Similarly, we may assume that $H_{3}^{(k),\Omega'}<G .$ Let $\{g_1, g_2,...,g_m\}$ be the set of right transversal of $H_3$ in $G$, where $g_1=1$. Write $$H:=H_{3}~, \Delta_{g_{i}H}=\{(\alpha,g_{i}H) ~:~ \alpha\in \Delta\},$$ and we denote $$\sigma_i: H \rightarrow H^{(k),\Delta_{g_{i}H}},\quad ~~h\mapsto g_{i}h{g_{i}}^{-1} .$$
Then $(H,\Delta)$ is permutation isomorphic to $(\sigma_i(H),\Delta_{g_{i}H})$.
Thus $(\sigma_1(H),\Omega')$ permutation isomorphism $(H,\Omega')$ and ${\sigma_1(H)}^{(k),\Omega'}=\sigma_1(H)$ by $H^{(k),\Omega'}=H$.

 Let $$\rho_i: H^{(k),\Delta_{H}} \rightarrow H^{(k),\Delta_{g_{i}H}},\quad by \quad(\alpha, {g_{i}H})^h= (\alpha^{\rho_i(h)}, {g_{i}H}),~~h\in H^{(k),\Delta_{H}}.$$ It is easy to see that $(H^{(k),\Delta_{H}},\Delta_{H})$ is permutation isomorphic to $(H^{(k),\Delta_{g_{i}H}},\Delta_{g_{i}H})$. It follows that $ H^{(k),\Delta_{H}}$ acts faithfully on $\Omega'$.
 Set $H':=\sigma_1(H)$. For $h\in {H'^{(k),\Delta_{H}}}$, then there exists $x\in  H'$ so that $$((\alpha_1,H),(\alpha_2,H),...,(\alpha_k, H))^h=((\alpha_1,H),(\alpha_2,H),...,(\alpha_k, H))^x$$ for all $\alpha_1,\alpha_2,...,\alpha_k\in \Delta$. So, for all $(\alpha_1, g_{i_1}H),(\alpha_2, g_{i_2}H),...,(\alpha_k, g_{i_k}H)\in \Omega'$, $$
\begin{aligned}
((\alpha_1, g_{i_1}H),...,(\alpha_k, g_{i_k}H))^{h} &=((\alpha_1^{\rho_{i_1}(h)}, g_{i_1}H),(\alpha_2^{\rho_{i_2}(h)}, g_{i_2}H),...,(\alpha_k^{\rho_{i_k}(h)}, g_{i_k}H)) \\
&=((\alpha_1^{\rho_{i_1}(x)}, g_{i_1}H),(\alpha_2^{\rho_{i_2}(x)}, g_{i_2}H),...,(\alpha_k^{\rho_{i_k}(x)}, g_{i_k}H))\\
&=((\alpha_1, g_{i_1}H),...,(\alpha_k, g_{i_k}H))^{x} .
\end{aligned}
$$ This implies that $h\in H'^{(k), \Omega'}=H'$. Hence ${H'^{(k),\Delta_{H}}}=H'$, and so $H^{(k),\Delta}=H$, which proves that $H$ is $k$-closed over $\Delta$ and so $H_{3}$ is totally $k$-closed. 
 \medskip

 For $3<i<n$, let $\Delta'$ be any set on which $H_{i}$ acts faithfully an $\Omega''=\Delta'\times G/H_{i}$. Then we may assume that $$H_{i}^{(k),\Omega''}=G~ ~\mathrm{or} ~~H_{i}^{(k),\Omega''}=H_{i}<G~~ \mathrm{or} ~~H_{i}^{(k),\Omega''}=H_{j'}<G$$ for some $i\leq j'\leq n$. 
 
  \medskip
{\bf subcase 2.1.}\, If $H_{i}^{(k),\Omega''}=H_{i}<G~\mathrm{or}~H_{i}^{(k),\Omega''}=H_{j'}<G$. We can conclude that $H_{i}$ is totally $k$-closed from the same reasoning above.
  \medskip

{\bf subcase 2.2.}\, If $H_{i}^{(k),\Omega''}=G$. Then $G$ acts trivial on $G/H_{i}$
 because $H_{i}^{(k),\Omega''}$ acts trivial on $G/H_{i}$.

\medskip 
By subcase 2.1 and subcase 2.2, it implies that $H_{i}$ is totally $k$-closed in any case. Therefore it suffices to assume that $H_{n-1}$ is totally $k$-closed. It follows that $H_{n-1}$ is an abelian $p$-group with $n(H_{n-1})\leq k-1$ by inductive hypothesis, which yields that $H$ is abelian $p$-group with $n(H)\leq k-1$ for all proper subgroup of $G$. Then  $$G=\langle a,b~:~a^{p^n}=b^{p^m}=1, a^b=a^{1+p^{n-1}},n\geq 2\rangle$$ or $$G=\langle a,b,c~:~a^{p^n}=b^{p^m}=c^p=1, [a,b]=c,[c,a]=[c,b]=1\rangle$$ by L. R$\acute{e}$dei in \cite{10}, which is a contradiction with $Z(G)$ being non-cyclic.
\medskip

The converse direction follows immediately from Lemma 3.4 and Lemma 3.5.

\end{proof}
\medskip

 \noindent{\bf Proof of Theorem 1.1.}\, By Theorem 3.1, we have $$G^{(k), \Omega}=\prod_{P \in \operatorname{Syl}(G)} P^{(k), \Omega}$$ and hence $G$ is totally $k$-closed if and only if $P$ is totally $k$-closed for all $P\in \mathrm{Syl}_p({G})$. Accordingly to Theorem 3.6 and  Lemma 3.4, this conclusion of Theorem 1.1 holds, as required.

\vskip 3mm
The proof of Theorem 1.1 is now complete.\\

$\mathbf{Acknowledgment}$
\medskip

 The authors have no conflict of interest to declare that are relevant to this article.

\end{document}